\documentclass[a4paper,12pt]{amsart}
\usepackage{amsmath,amssymb,amsfonts,amsthm}
\usepackage{bbm}
\usepackage{mathrsfs}

\newtheorem{thm}{Theorem}[section]
\newtheorem*{thm*}{Theorem}
\newtheorem{lem}[thm]{Lemma}
\newtheorem*{lem*}{Lemma}

\newtheorem*{cor*}{Corollary}

\newtheorem*{prop*}{Proposition}
\theoremstyle{definition} 
\newtheorem{defn}[thm]{Definition}
\newtheorem*{defn*}{Definition}
\theoremstyle{remark}

\newtheorem*{rem*}{Remark}

\newtheorem*{example*}{Example}

\newtheorem*{que*}{Question}

\newcommand{\N}{\mathbb N}

\newcommand{\Z}{\mathbb Z}

\newcommand{\norm}[1]{\left\Vert#1\right\Vert}
\newcommand{\abs}[1]{\left\vert#1\right\vert}
\newcommand{\set}[1]{\left\{#1\right\}}
\newcommand{\eps}{\varepsilon}

\newcommand{\ind}{\mathbbm{1}}

\newcommand{\F}{\mathcal{F}}

\renewcommand{\epsilon}{\varepsilon}

\renewcommand{\leq}{\leqslant}
\renewcommand{\geq}{\geqslant}

\title{A comment on ergodic theorem for amenable groups}
\author{Bartosz Frej}
\author{Dawid Huczek}
\thanks{Research of both authors is supported from resources for science in years 2013-2018 as research project (NCN grant 2013/08/A/ST1/00275, Poland)}

\begin{document}
\begin{abstract}
We prove a version of ergodic theorem for an action of an amenable group, where a F\o lner sequence needs not to be tempered. Instead, it is assumed that a function satisfies certain mixing condition.
\end{abstract}

\maketitle

\renewcommand{\thethm}{\arabic{thm}}

In \cite{L} E.Lindenstrauss proved the ergodic theorem for actions of amenable groups, which is commonly used. The F\o lner sequence along which ergodic averages converge to an invariant function must satisfy the condition of being tempered. Since every F\o lner sequence has a subsequence which is tempered, the theorem is sufficient for many applications. In \cite{BDM} it was shown that for the Bernoulli groups shift the assumption that the F\o lner sequence is tempered may be relaxed if one considers frequency of visits in a cylinder set. The aim of the current paper is to push further in this direction and ivestigate in what circumstances temperedness is not necessary.

Let $G=\{g_1,g_2,...\}$ be an amenable group and $(F_n)_{n\in\N}$ a F\o lner sequence in $G$. Assume that for every $\alpha\in[0,1)$ the series $\sum_{n=1}^\infty \alpha^{|F_n|}$ converges. Since it is already satisfied if $|F_n|$ strictly increases, the assumption is much weaker than temperedness. 
	For a finite nonempty set $F\subset G$ and any set $S\subset G$ we denote
	\[
	\underline d_F(S)=\inf_{g\in G} \frac{|S\cap Fg|}{|F|} \qquad \textrm{and}\qquad \overline 
	d_F(S)=\sup_{g\in G} \frac{|S\cap Fg|}{|F|}.
	\]
	The \emph{lower} and \emph{upper} \emph{Banach densities} of 
	$S$ are defined by formulas:
		\begin{align*}
		\underline d(S) &= \sup\{\underline d_F(S): F\subset G, F \text{ is 
		finite}\},\\
		\overline d(S) &= \,\inf\,\{\overline d_F(S): F\subset G, F \text{ is 
		finite}\}.
	\end{align*}
	If $(F_n)$ is a F\o lner sequence then we also have
	\[
	\underline d(S)=\lim_{n\to\infty} \underline d_{F_n}(S) \ \ \ \text{ and 
	} \ \ \ \overline d(S)=\lim_{n\to\infty} \overline d_{F_n}(S).
	\]
	The following lemma was proved in \cite{BDM}.
\begin{lem}	\label{partition}
For every finite set $K\subset G$ and $\delta>0$ there exists a partition $D_0,D_1,...,D_r$ of $G$ such that 
\begin{enumerate}
	\item $\overline{d}_{(F_n)}(D_0) \leq \delta$,
	\item $\underline{d}_{(F_n)}(D_i) >0$ for every $i=1,...,r$,
	\item for every $i=1,...,r$, if $g,h\in D_i$ then $Kg \cap Kh = \varnothing$.
\end{enumerate}
\end{lem}

Let $G$ act via measure preserving transformations on a probability space $(X,\mu)$. 
To simplify the notation we will indentify $G$ with the related group of automorphisms and write $gx$ for the outcome of the action of an automorphism associated to $g$ on $x$, and $f\circ g$ for the composition of a function $f$ and the automorphism.
According to \cite{AdJ}, for any action of $\Z$ one can find a F\o lner sequence (with cardinalities of sets increasing slowly) such that the ergodic averages with respect to that sequence fail to converge for some function. Therefore, some constraints must be put on the function, whose ergodic averages we study. For $\F\subset L^1(\mu)$ let $\sigma(\F)$ denote the smallest sub-$\sigma$-algebra with respect to which every $f\in\F$ is measurable.
\begin{defn}
We will say that $f$ is \emph{$\eps$-independent from a sub-$\sigma$-algebra $\Sigma_0$} if for every $B\in\Sigma_0$ of positive measure it holds that 
\begin{equation}	\label{mix}
\abs{\int_B f d\mu_B - \int f d\mu} < \eps,
\end{equation}
where $\mu_B$ is the conditional measure on $B$.
\end{defn} 
\noindent A set $A$ is $\eps$-independent from $\Sigma_0$ if its charachteristic function $\ind_A$ is, i.e., for all $B\in\Sigma_0$ such that $\mu(B)>0$,
\[
\abs{\mu_B(A)-\mu(A)}<\eps
\]
or, in other words,
\[
|\mu(A \cap B) - \mu(A)\mu(B)| <\eps\mu(B).
\]

\begin{thm} \label{main}
Let $f\in L^\infty(\mu)$ be such that for every $\eps>0$ there exists a finite set $K\subset G$ such that $f$ is $\eps$-independent from $\sigma(\{f\circ g: g\not\in K\})$.
Then
\[
\lim_{n\to\infty} \frac1{|F_n|} \sum_{g\in F_n} f(gx) = \int f d\mu \qquad \mu\textrm{-a.e.}
\]
\end{thm}
\noindent Note that the neutral element of $G$ belongs to $K$ for non-constant $f$.

Before the proof let us recall the following concentration inequality.
\begin{thm}[Azuma-Hoeffding inequality]
Suppose that $(M_n)_{n\in\N}$ is a martingale on $(X,\mu)$ such that $M_0=0$, $EM_n=0$ for all $n\in\N$ and $|M_k-M_{k-1}|\leq d_k$ almost surely for some constants $d_k$.
Then, for every $\eps>0$
\[
\mu\left(\set{x:|M_n(x)|>\eps}\right) \leq 2\exp\left(\frac{-\eps^2}{2\sum_{k=1}^n d_k^2}\right)
\]
\end{thm}
\begin{proof}[Proof of thm.\ref{main}]
Fix a function $f\in L^\infty(\mu)$ and a number $\eps>0$ and let $K$ be as in the assumption, chosen for $\eps/2$. Using lemma \ref{partition}, choose a partition $D_0,D_1,...,D_r$ for $\delta=\eps/(7\norm{f}_\infty)$.
For $H\subset G$, $|H|<\infty$, define
\begin{gather*}
Y_H=\sum_{g\in H} f\circ g \qquad
\overline{Y}_H = \frac1{|H|} Y_H
\end{gather*}
Then $\int f d\mu = E\overline{Y}_H$.

Let $H^{(n,i)}_k=\{g_{i_1},...,g_{i_k}\}$ be a set of $k$ first elements of $G$ belonging to $H^{(n,i)}=F_n \cap D_i$.
Let $\F^{(n,i)}_k=\sigma(\{f\circ g:g\in H^{(n,i)}_k\})$. 
Clearly, $\{\F^{(n,i)}_k\}$ is a finite filtration and for each pair $(n,i)$ the process $Y^{(n,i)}_k=Y_{H^{(n,i)}_k}$ is adapted to it. Let $\F_0^{(n,i)}$ be the trivial $\sigma$-algebra in $X$ and let $Y_0^{(n,i)}=0$ a.s. By Doob's decomposition,
\[
Y^{(n,i)}_k = M_k^{(n,i)} + N_k^{(n,i)},
\]
where 
\[
M_k^{(n,i)}=\sum_{j=1}^k (Y_j^{(n,i)}-E(Y_j^{(n,i)}|\F_{j-1}^{(n,i)}))
\]
is a martingale and
\begin{eqnarray}
N_k^{(n,i)}&=&\sum_{j=1}^k \left(E(Y_j^{(n,i)}|\F_{j-1}^{(n,i)})-Y_{j-1}^{(n,i)}\right) \nonumber \\
&=&\sum_{j=1}^k E\left(f\circ g_{i_j}|\F_{j-1}^{(n,i)}\right)	\label{to_estimate}
\end{eqnarray}
is a predictable process, i.e. each $N_k^{(n,i)}$ is $\F_{k-1}^{(n,i)}$-measurable.
Then $EM_k^{(n,i)}=0$ and $EN_k^{(n,i)}=k\int f d\mu$. Note also, that
\begin{eqnarray*}
|M_k^{(n,i)}-M_{k-1}^{(n,i)}| &\leq&  |Y_k^{(n,i)}-Y_{k-1}^{(n,i)}|+|N_k^{(n,i)}-N_{k-1}^{(n,i)}| \\
& =& |f\circ g_{i_k}| + |E(f\circ g_{i_k}|\F_{k-1}^{(n,i)})| \leq 2||f||_\infty.
\end{eqnarray*}
Similarly, we write $Y_{H^{(n,i)}} = M^{(n,i)} + N^{(n,i)}$ and we denote:
\[
\overline{M}^{(n,i)} = \frac1{|H^{(n,i)}|}M^{(n,i)}, \qquad \overline{N}^{(n,i)} = \frac1{|H^{(n,i)}|}N^{(n,i)}
\]

We will estimate the following quantity:
\begin{multline*}
\mu\left(\set{x:\abs{\overline{Y}_{F_n}(x)-E\overline{Y}_{F_n}}>\eps }\right)\\
\leq \mu\left(\set{x:\sum_{i=0}^r \frac{\abs{H^{(n,i)}}}{|F_n|} \abs{\overline{Y}_{H^{(n,i)}}(x)-E\overline{Y}_{H^{(n,i)}}}>\eps}\right)\\
\leq \mu(\mathcal{M}_n) + \mu(\mathcal{N}_n),
\end{multline*} 
where
\[
\mathcal{M}_n=\set{x:\begin{split}&\frac{\abs{H^{(n,0)}}}{|F_n|} \abs{\overline{Y}_{H^{(n,0)}}(x)-E\overline{Y}_{H^{(n,0)}}} +\\&+ \sum_{i=1}^r \frac{\abs{H^{(n,i)}}}{|F_n|} \abs{\overline{M}^{(n,i)}(x)-E\overline{M}^{(n,i)}}>\eps/2 \end{split}}
\]
and
\[
\mathcal{N}_n=\set{x:\sum_{i=1}^r \frac{\abs{H^{(n,i)}}}{|F_n|} \abs{\overline{N}^{(n,i)}(x)-E\overline{N}^{(n,i)}}>\eps/2}
\]

Below we estimate the second summand.
For each $j$,
\[
\abs{E\left(f\circ g_{i_j}|\F_{j-1}^{(n,i)}\right) - \int f d\mu} < \eps/2
\]
almost surely, because $f\circ g_{i_j}$ is $\eps/2$-independent from $\F_{j-1}^{(n,i)}$. Then
\[
\abs{\left(\overline{N}^{(n,i)}(x)-E\overline{N}^{(n,i)}\right)} \leq \frac1{H^{(n,i)}}\sum_{j\in H^{(n,i)}}\abs{E\left(f\circ g_{i_j}|\F_{j-1}^{(n,i)}\right) - \int f d\mu} < \eps/2
\]
almost surely.
Hence also the average satisfies
\[
\sum_{i=1}^r \frac{\abs{H^{(n,i)}}}{|F_n|} \abs{\overline{N}^{(n,i)}(x)-E\overline{N}^{(n,i)}} < \eps/2,
\]
so $\mu(\mathcal{N}_n)=0$.

Now we will estimate $\mu(\mathcal{M}_n)$. We have $\overline{d}_{(F_n)}(D_0) \leq \eps/7\norm{f}_\infty$, so for large~$n$,  
\[
\frac{\abs{H^{(n,0)}}}{|F_n|}<\eps/6\norm{f}_\infty
\]
Clearly, $Y_{H^{(n,0)}}$ is a sum of functions with bounded range, hence
\[
\frac{\abs{H^{(n,0)}}}{|F_n|}\abs{\overline{Y}_{H^{(n,0)}}(x)-E\overline{Y}_{H^{(n,0)}}} \leq \frac{\abs{H^{(n,0)}}}{|F_n|}\cdot2\norm{f}_\infty < \eps/3
\]
for large $n$.

By Azuma-Hoeffding inequality, for each $i=1,...,r$ it holds that 
\begin{multline*}
\mu\left(\set{x:\abs{\overline{M}^{(n,i)}(x)-E\overline{M}^{(n,i)}}>\eps/6}\right) =\\
= \mu\left(\set{x:\abs{M^{(n,i)}(x)-EM^{(n,i)}}>\eps/6\abs{H^{(n,i)}}}\right) \leq 2\gamma^{\abs{H^{(n,i)}}}
\end{multline*}
for $\gamma=\exp\left(-\frac{\eps^2}{8||f||_\infty^2}\right)<1$ (recall that $|M_k^{(n,i)}-M_{k-1}^{(n,i)}|$ is bounded almost surely). 
Denoting 
\[
X_\eps=\set{x:\exists i=1,...,r\ \abs{\overline{M}^{(n,i)}(x)-E\overline{M}^{(n,i)}}>\eps/6}
\]
we obtain
\[
\mu(X_\eps)\leq 2r\cdot \gamma^{\min_i|H^{(n,i)}|}. 
\]
Thus, outside $X_\eps$ not only $\abs{\overline{M}^{(n,i)}(x)-E\overline{M}^{(n,i)}}\leq\eps/6$, but also the weighted average satisfies the same inequality
\[
\sum_{i=1}^r \frac{\abs{H^{(n,i)}}}{|F_n|} \abs{\overline{M}^{(n,i)}(x)-E\overline{M}^{(n,i)}}\leq\eps/6,
\]
so it is only the set $X_\eps$ on which it may happen that
\[
\frac{\abs{H^{(n,0)}}}{|F_n|} \abs{\overline{Y}_{H^{(n,0)}}(x)-E\overline{Y}_{H^{(n,0)}}} + \sum_{i=1}^r \frac{\abs{H^{(n,i)}}}{|F_n|} \abs{\overline{M}^{(n,i)}(x)-E\overline{M}^{(n,i)}}>\eps/2
\]
By positive density of each $D_i$, there is a positive number $\beta$ such that $\min_i|H^{(n,i)}| > \beta |F_n|$ for large $n$.
We obtain 
\[
\mu(\mathcal{M}_n) \leq 2r\cdot \gamma^{\min_i|H^{(n,i)}|} \leq 2r\left(\gamma^\beta\right)^{|F_n|} =: 2r\alpha^{|F_n|}
\]
By the assumption on the F\o lner sequence $(F_n)$, the series $\sum_n \alpha^{|F_n|}$ converges. By Borel-Cantelli lemma we obtain that 
\[
\mu\left(\limsup_n \mathcal{M}_n\right)=0,
\]
so
\[
\begin{split}
\mu\left(\set{x:\forall N\ \exists n\geq N \ \abs{\overline{Y}_{F_n}(x)-\int f d\mu}>\eps}\right) \leq \\
\leq\mu\left(\limsup_n \mathcal{M}_n\right) + \mu\left(\limsup_n \mathcal{N}_n\right)= 0,
\end{split}
\]
Finally, after taking an appropriate countable intersection we obtain
\[
\mu\left(\set{x:\lim_n \overline{Y}_{F_n}(x)=\int f d\mu}\right) = 1.
\]
\end{proof}

\end{document}